\documentclass[12pt]{amsart}

\usepackage{url} 
\usepackage{charter}
\usepackage{macros}

\standardsettings

\usepackage{setspace}
\draftfalse

\newcommand{\alphabet}{E}

\DeclareMathOperator{\DS}{DimSpec}

\begin{document}
\title{On the dimension spectra of infinite conformal iterated function systems}

\authortushar
\authordavid


\begin{abstract}
The dimension spectrum of a conformal iterated function system (CIFS) is the set of all Hausdorff dimensions of its various subsystem limit sets. This brief note provides two constructions -- (i) a compact perfect set that cannot be realized as the dimension spectrum of a CIFS; and (ii) a similarity IFS whose dimension spectrum has zero Hausdorff dimension, and thus is not uniformly perfect -- which resolve questions posed by Chousionis, Leykekhman, and Urba\'nski, and go on provoke fresh conjectures and questions regarding the topological and metric properties of IFS dimension spectra.
\end{abstract}
\maketitle

\section{Introduction}

The study of iterated function systems (IFSes), which began in earnest in the early 1980s, increased in popularity during the renaissance following Benoit Mandelbrot's seminal work \emph{Les objets fractals} \cite{Mandelbrot_book_1975} and his invention of the word \emph{fractal} to describe ``a mathematical set or concrete object whose form is extremely irregular and/or fragmented at all scales''. 
Leading researchers who worked on IFS theory and developed a variety of extensions include Bandt, Barnsley, Dekking, Falconer, Graf, Hata, Hutchinson, Mauldin, Schief, Simon, Solomyak, and Urba\'nski -- see \cite{BandtGraf, Barnsley, Dekking, Falconer_book1990, GMW, Hutchinson, MauldinUrbanski1, MauldinUrbanski2, Schief, SSU} for a small sample of their seminal research. For applications in engineering and science, e.g. in computer graphics, image processing, wavelets, probabilistic growth models, stochastic dynamical systems, etc. -- see \cite{BBDFK, DLLT, EPSE, Jorgensen2, KLMV, LL}.

Several pioneering works focussed on IFSes consisting of finitely many Euclidean similarities; afterwards the theory was extended to handle systems with infinitely many maps (called infinite IFSes) that were conformal. Mauldin and Urba\'nski were among the pioneers of this extension of IFS theory, first to the study of infinite conformal iterated function systems (CIFSes), and then to their generalizations, viz. conformal graph directed Markov systems (CGDMSes), see \cite{MauldinUrbanski1, MauldinUrbanski2}.
CIFS and CGDMS limit sets model several among the intensively studied fractals arising from either side of Sullivan's dictionary \cite{McMullen_classification, Sullivan_conformal_dynamical} (see also \cite[Table 1]{DSU_rigidity}): namely, certain Julia sets associated with holomorphic and meromorphic iteration, as well as certain Fuchsian and Kleinian limit sets associated with actions of discrete subgroups of isometries of hyperbolic (negatively curved) spaces. 

In particular, the CIFS framework is perfectly suited to encode various sets that appear naturally at the interface of dynamical systems, fractal geometry and Diophantine approximation. For instance, one can encode real numbers via their continued fraction expansions and this leads to the Gauss continued fraction IFS, which is a prime example of an infinite CIFS that comprises of the M\"obius maps $x \mapsto 1/(a + x)$ for $a \in \N$. Given any subset $A\subset \N$, let $\Lambda_A$ denote the set of all irrationals $x\in [0,1]$ whose continued fraction partial quotients all lie in $A$. Then $\Lambda_A$ may be expressed as the limit set of the subsystem of the Gauss IFS that comprises the maps $x \mapsto 1/(a + x)$ for $a \in A$, see \cite{Hensley_book}. The elements of $\Lambda_A$ when $A$ is a finite set are known as {\it bounded type} numbers. For instance, studying $E_{\{1,2\}}$ relates to the problem of finding rational numbers of given denominator having all partial quotients equal to $1$ or $2$. 
The Hausdorff dimensions of such sets have continued to be intensely investigated since several decades, see e.g. \cite{Bumby, MauldinUrbanski4, JenkinsonPollicott, Moreira, Moreira2} and the references therein.

It was conjectured independently by Hensley \cite{Hensley2} and by Mauldin--Urba\'nski \cite{MauldinUrbanski4} in the 1990s that the set $\{ \HD(\Lambda_A) : A \subset \N ~\text{with}~ |A| < \infty \}$ is a dense subset of $[0,1]$. This conjecture, dubbed the {\it Texan conjecture} by Jenkinson \cite{Jenkinson}, was resolved in the affirmative in 2006 by Kesseb\"ohmer--Zhu \cite{KessebohmerZhu}. Research surrounding the Texan conjecture gave birth to the study of topological and metric features of the dimension spectrum of an infinite CIFS. 

Understanding the geometry and topology of IFS and GDMS dimension spectra has since presented researchers with several challenges, see \cite{CLU,CLU2,GMR,Ghenciu,Ghenciu2,Jurga}. In their recent papers \cite{CLU,CLU2} Chousionis--Leykekhman--Urba\'nski (CLU) leverage the thermodynamic formalism to commence a careful study of the dimension spectra of finitely irreducible CGDMSes as well as for continued fractions with coefficients restricted to infinite subsets of natural numbers. In particular, they provide a positive answer to the analogue of the Texan conjecture for complex continued fractions \cite[Theorem 1.4]{CLU}. CLU proved that the dimension spectrum of every infinite CIFS satisfying the open set condition is compact and perfect, and conjectured that every such set may be realized as the dimension spectrum of a similarity IFS. They also asked whether there exists an IFS whose dimension spectrum is not uniformly perfect. This short note resolves both these questions posed by CLU in \cite{CLU} regarding the dimension spectrum of infinite CIFSes, and concludes with some fresh conjectures and research directions in this seam. 

\

\textbf{Conventions.} In this note $\N := \{ 1,2,3,\dots\}$. We write $x \asymp y$ to mean that $x$ and $y$ are multiplicatively comparable, i.e. there exists $C>0$ such that $1/C \leq x/y \leq C$. We use $\Theta(x)$ to denote any positive quantity multiplicatively comparable to $x$. We use $x \gtrgtr y$ to mean that for every $c>1$ we have that $x$ is eventually bigger than $cy$. We write $\HD(X)$ to denote the Hausdorff dimension of a set $X \subset \R^n$; and write $\rho(T)$ to denote the spectral radius of an operator $T$. We simplify notation by writing $\HD(A) := \HD(\Lambda_A)$ for the Hausdorff dimension of subsystem limit sets (see Definitions \ref{definitionCIFS} and \ref{subsystem}).

\

\textbf{Acknowledgements.} 
The authors discussed the results of this research with Vasileios Chousionis and Mariusz Urba\'nski at the American Institute of Mathematics (AIM) in March 2018, where they were collaborating via the \href{https://aimath.org/programs/squares/}{SQuaREs program}. We thank the AIM staff for nurturing this outstanding collective research opportunity and for providing us with excellent working conditions. We thank Bal\'azs B\'ar\'any for stimulating discussions. We also thank an anonymous referee for their valuable suggestions that helped clarify certain infelicities as well as improve the exposition of our proofs.  David Simmons is also supported by a 2018 Royal Society University Research Fellowship, URF\tbs R1\tbs180649.

\draftnewpage
\section{Definitions and statement of results}

The definition of a CIFS appears in several places in the literature, see e.g. \cite[Remark 3.2]{CLU}\footnote{The original may be found in \cite[\62, p.6-7]{MauldinUrbanski1}.}.

\begin{definition}[]
\label{definitionCIFS}
Fix $d\in\N$. A collection $\UU = (u_a)_{a\in \alphabet}$ of self-maps of $\R^d$ is a \textsf{\textbf{conformal iterated function system (CIFS)}} on $\R^d$ if:
\begin{enumerate}[1.]
\item $\alphabet$ is a countable (finite or infinite) index set, which is referred to as an \textsf{\textbf{alphabet}};
\item $X\subset\R^d$ is a nonempty compact set which is equal to the closure of its interior;
\item For all $a\in \alphabet$, $u_a(X) \subset X$;
\item $V\subset\R^d$ is an open connected bounded set such that $\mathrm{dist}(X,\R^d\butnot V) > 0$;
\item For each $a\in \alphabet$, $u_a$ is a conformal homeomorphism from $V$ to an open subset of $V$;
\item (\textsf{\textbf{Cone condition}})
$
\inf_{x\in X} \inf_{r\in (0,1)} \lambda(X\cap B(x,r))/r^d > 0,
$
where $\lambda$ denotes the Lebesgue measure on $\R^d$;

\item (\textsf{\textbf{Open set condition (OSC)}}) For all $a\in \alphabet$ the collection $(u_a(\Int(X)))_{a\in \alphabet}$ is disjoint;
\item (\textsf{\textbf{Uniform contraction}}) $\sup_{a\in \alphabet} \sup |u_a'| < 1$, and if $\alphabet$ is infinite, $\lim_{a\in \alphabet} \sup |u_a'| = 0$;
\item (\textsf{\textbf{Bounded distortion property}}) For all $n\in\N$, $\omega\in \alphabet^n$, and $x,y\in V$,
$|u_\omega'(x)| \asymp |u_\omega'(y)|$.
\end{enumerate}
\end{definition}

\begin{definition}
\label{definitionlimitset}
Given a countable alphabet $E$ as above, we denote by $E^n$ the set of all words of length $n$ formed using this alphabet, and by $E^*$ the set of all finite words formed using this alphabet. In other words,  
\[
E^* = \bigcup_{n=0}^{\infty} E^n .
\]
If $\omega\in E^*\cup E^\N$, i.e. $\omega$ is either a finite or infinite word formed using the alphabet $E$, then we denote subwords of $\omega$ by
\[
\omega_{n + 1}^{n + r} := (\omega_{n + i})_{i = 1}^r \in E^r.
\]
If $\omega\in E^n$ is a finite word then we define
\[
u_\omega(x) := u_{\omega_1}\circ\ldots\circ u_{\omega_n}(x).
\]
The \textsf{\textbf{coding map}} of the CIFS $\UU = (u_a)_{a\in \alphabet}$ is the map $\pi:\alphabet^\N\to X$ defined by the formula
\[
\pi(\omega) 
=\lim_{n\rightarrow\infty}u_{\omega_1^n}(x_0)
= \lim_{n\to\infty} u_{\omega_1}\cdots u_{\omega_n}(x_0),
\]
where $x_0\in X$ is an arbitrary point. By the Uniform Contraction hypothesis, $\pi(\omega)$ exists and is independent of the choice of $x_0$. The \textsf{\textbf{limit set}} of the CIFS is the image of $\alphabet^\N$ under the coding map, and will be denoted by $\Lambda = \Lambda_E :=\pi(E^\N)$.
\end{definition}

Note that the \textsf{\textbf{uniform contraction}} hypothesis implies that the coding map is always H\"older continuous, assuming that the metric on $\alphabet^\N$ is given by the formula
\[
\dist(\omega,\tau) = \lambda^{|\omega\wedge\tau|},
\]
where $\lambda\in (0,1)$ and $\omega\wedge\tau$ is the longest word which is an initial segment of both $\omega$ and $\tau$.

The class of CIFSes consisting of similarities has been studied particularly intensively. We give the definition below in the basic case when $d=1$.

\begin{definition}
\label{definitionIFS}
Let $E$ be a countable alphabet, as above. A  \textsf{\textbf{similarity iterated function system (SIFS)}} on $\R$ is a uniformly contracting and uniformly bounded collection of similarities $\UU = \UU_E :=(u_a)_{a\in E}$ indexed by $E$. We write each similarity $u_a:\R\to\R$ as 
$
u_a(x) = \lambda_a x + b_a
$
for $0< |\lambda_a| < 1 $ and $b_a\in \R$.
A collection $(u_a)_{a\in E}$ of similarities is uniformly contracting or uniformly bounded if
\[
\sup_{a \in E} |\lambda_a| <1 ~\text{or}~ \sup_{a \in E} |b_a| < \infty, 
\]
respectively. To guarantee that our SIFS $\UU$ is a CIFS as defined above, we assume that $\UU$ satisfies the \textsf{\textbf{open set condition (OSC)}}, i.e. there exists an open $W\subset\R$, whose closure satisfies the cone condition, such that the collection $(u_a(W))_{a\in E}$ is a disjoint collection of subsets of $W$. 
Note that the OSC assumption implies that the collection of similarities are uniformly bounded, i.e. that $\sup_{a \in E} |b_a| < \infty$, and also that $\lim_{a\in E} |\lambda_a| = 0$ (by taking the Lebesgue measure of the inclusion $\bigcup_{a\in E} u_a(W) \subset W$).
 
As above, the \textsf{\textbf{limit set}} of $\UU = \UU_E$ is the image of the \textsf{\textbf{coding map}} $\pi:E^\N\to\R$ defined by
\[
\pi(\omega) 
=\lim_{n\rightarrow\infty}u_{\omega_1^n}(x_0)
= \lim_{n\to\infty} u_{\omega_1}\cdots u_{\omega_n}(x_0),
\]
and will be denoted $\Lambda = \Lambda_E :=\pi(E^\N)$. Note that given any SIFS (not necessarily satisfying the OSC) the uniformly contracting and uniformly bounded condition implies that $\pi$ is defined. 
When we write $\UU$ is an SIFS, we assume as is common \cite[Remark 3.2]{CLU}, that the OSC is satisfied.
\end{definition}

\begin{definition}
\label{subsystem}
Given an SIFS or CIFS $\UU = \UU_E$ we will be interested in sub-CIFSes or sub-SIFSes (called \textsf{\textbf{subsystems}}) formed by restrictions of $\UU$ to various subsets of the original alphabet $E$. Given $A\subset E$, the corresponding subsystem, coding map, and limit set are denoted by $\UU_A$, $\pi_A$, and $\Lambda_A$, respectively. 
\end{definition}

\begin{definition}
The \textsf{\textbf{(Hausdorff) dimension spectrum}} of a CIFS $\UU = (u_a)_{a\in E}$ is defined as
\[
\DS(\UU) := \{ \HD(\Lambda_A) : A \subset E \}.
\]
\end{definition}

CLU proved \cite[Theorem 1.2]{CLU} that the dimension spectrum of an infinite CIFS is compact and perfect. 
They went on to conjecture \cite[Conjecture 1.3]{CLU} that every compact perfect set $K \subset [0, \infty)$ can be the dimension spectrum of a CIFS. Note that by taking a one-element subset of the alphabet, we get a subsystem whose limit set is a singleton and thus of Hausdorff dimension zero. Thus $0 \in \DS(\UU)$ for all iterated function systems $\UU$. Thus their original conjecture should be reformulated to only consider compact perfect sets containing zero. Our first result shows that their (reformulated) conjecture was too optimistic:

\begin{theorem}
\label{theoremcompactperfect}
There exists a compact and perfect set $K \subset [0, 1]$ such that $0 \in K$ and $\DS(\UU) \neq K$ for all CIFSes $\UU$ on $\R$.
\end{theorem}
\noindent The proof of Theorem \ref{theoremcompactperfect} shows that it remains true if ``$\R$'' is replaced by ``$\R^d$'' for any $d$.

CLU recognized that their conjecture had ``room for many partial results and open questions''. They asked, in particular, whether there exists an IFS whose dimension spectrum is not uniformly perfect. Our second result answers their question in the affirmative.

\begin{theorem}
\label{theoremHDzero}
There exists an infinite SIFS on $\R$ whose dimension spectrum has Hausdorff dimension zero. In particular, the dimension spectrum is not uniformly perfect.
\end{theorem}

\draftnewpage
\section{Proof of Theorem \ref{theoremcompactperfect}}
\label{section}

\begin{center}
{\it To simplify notation in the sequel, we will write  $\HD(A) := \HD(\Lambda_A)$ for all $A\subset E$.}
\end{center}

Let $A = \{0,1\}$, and let $(\sigma_n)_{n \geq 1}$ be an enumeration of $A^*$ such that the map $n \mapsto |\sigma_n|$ is nondecreasing.
Next, let $g(\sigma_n) := 4^{-n!}$, and let $f: A^\N \to [0,1]$ be defined by 
\[
f(\omega) := \sum_{n : \: \omega_{n+1}=1} g(\omega_1^n).
\] 
We define $K := f(A^\N)$. 
The map $f$ is continuous, and thus $K$ is compact.

The map $f$ is injective. Indeed, suppose  $\omega, \tau \in A^\N$ are distinct. Without loss of generality we may take them to be of the form $\omega = \omega_1^n 0 \omega_{n+2}^\infty$ and $\tau = \omega_1^n 1 \tau_{n+2}^\infty$. It follows from the properties of $(\sigma_n)_{n \geq 1}$ and $g$ that $g(\omega_1^{n+k}) \leq 4^{-k} g(\omega_1^n)$ for all $k$.
Therefore we have that
\[
f(\tau) - f(\omega)  
\geq
g(\omega_1^n) - \sum_{m>n} g(\omega_1^m)
\geq g(\omega_1^n) [1 - \sum_{k>0} 4^{-k}] > 0.
\]
Since $K$ is the continuous injective image of a perfect space, it itself is thus perfect. Note that this calculation shows that $f(\tau) - f(\omega) \asymp g(\omega \wedge \tau)$.

We now want to prove that $\DS(\UU) \neq K$ for all all CIFSes $\UU$ on $\R$. So let $\UU = (u_a)_{a\in E}$ be an infinite CIFS on $\R$ with alphabet $E$, and by way of contradiction suppose that $\DS(\UU) = K$. Next, for a given $F \subseteq E$ with $2 \leq \#(F) < \infty$, we estimate how much the Hausdorff dimension of $\HD(F)$ increases after we add an extra symbol $b \in E \setminus F$. 

\begin{claim}
\label{claim4points}
For every $F \subseteq E$ with $2 \leq \#(F) < \infty$ and for every $b \in E \setminus F$, we have
\[
\HD(F \cup\{b\}) = \delta + \Theta(D_b^{\delta}),
\] 
where $\delta = \HD(F)$ and $D_b = \|u_b'\| := \sup  | u_b' |$, and the implied constant of $\Theta$ depends on $F \subseteq E$.
\end{claim}
\begin{proof}
Recall from Definition \ref{definitionCIFS} of a CIFS that $X\subset\R^d$ is a nonempty compact set which is equal to the closure of its interior. Let $C(X)$ denote the Banach space of continuous functions from $X$ to $(0,\infty)$, and let $L : C(X) \to C(X)$ denote the Perron--Frobenius operator of the CIFS $(u_a)_{a\in F}$, i.e.
\[
L f(x) := \sum_{a\in F} |u_a'(x)|^{\delta} f\circ u_a(x).
\]
Then there exists a positive continuous map $g : X \to (0,\infty)$ such that $L g = g$, by \cite[Theorem 6.1.2]{MauldinUrbanski2}.
Now let
\[
L' f(x) := L f(x) + |u_b'(x)|^{\delta} f\circ u_b(x).
\]
Then the logarithm of the spectral radius of $L'$ is $\log\rho(L') = P(F\cup\{b\},\delta)$, the pressure of the CIFS $(u_a)_{a\in F\cup \{b\}}$ evaluated at $\delta$, see \cite[Theorem 2.4.3, Theorem 2.4.6, and p.29]{MauldinUrbanski2}. To estimate this, we compute $L' g$. Now by the bounded distortion property and since $g$ is bounded from above and below on the compact set $X$, we have
\[
(L' - L) g (x) = |u_b'(x)|^{\delta} g\circ u_b(x) \asymp D_b^{\delta} g(x)
\]
and thus
\[
L' g = (1 + \Theta(D_b^{\delta})) g.
\]
Since $g$ and $L'$ are both positive, this tells us that the spectral radius satisfies
\[
\exp( P(F\cup\{b\},\delta) ) = \rho(L') = 1 + \Theta(D_b^{\delta}).
\]
Thus we have that 
\[
P(F\cup\{b\},\delta) \asymp D_b^{\delta}.
\]

On the other hand, from Bowen's formula we know that $P(F\cup\{b\},s) = 0$ where $s = \HD(F\cup\{b\})$, see \cite[Theorem 4.2.11]{MauldinUrbanski2}. Moreover, the negative derivative of pressure satisfies 
\begin{equation}
\label{negativepressurebounds}
-P'(F\cup \{b\},\cdot) \asymp 1.
\end{equation}
on $\CO{\delta}\infty$ independent of $b$. Indeed, the lower bound in \eqref{negativepressurebounds} follows from direct calculation, while the upper bound in \eqref{negativepressurebounds} follows from the convexity of pressure \cite[Proposition 4.2.8(b)]{MauldinUrbanski2} together with the fact that $P(F\cup\{b\},0) = \log(\#(F) + 1)$ is independent of $b$, and that $\delta > 0$ since $\#(F) \geq 2$. It thus follows that $s = \delta + \Theta(D_b^{\delta})$. This concludes the proof of Claim \ref{claim4points}.
\end{proof}

Next, we consider $F_1,F_2 \subset E$ such that $2 \leq |F_i| < \infty$, $\delta_i := \HD(F_i) > 0$, and $\delta_2 > \delta_1$; for instance, we could take $F_1 = \{a_1,a_2\}$ and $F_2 = \{a_1,a_2,a_3\}$ where $(a_n)_{n \geq 1}$ is an enumeration of $E$.

Now fix some $b \in E\setminus(F_1 \cup F_2)$. Since we assumed $\DS(U) = K = f(A^\N)$, there exist $\omega,\omega',\tau,\tau' \in A^\N$ such that 
\begin{align*}
f(\omega) &= \HD(F_1) = \delta_1&
f(\omega') &= \HD(F_1\cup\{b\})\\
f(\tau) &= \HD(F_2) = \delta_2&
f(\tau') &= \HD(F_2\cup\{b\}).
\end{align*}
By two applications of Claim \ref{claim4points} we have
\[
f(\omega') - f(\omega) \asymp D^{f(\omega)}
~~~\text{and}~~~
f(\tau') - f(\tau) \asymp D^{f(\tau)},
\]
where $D = D_b = \|u_b'\|$.
Now let $\omega'' := \omega\wedge\omega'$ and $\tau'' := \tau\wedge\tau'$. Then 
\begin{align*}
g(\omega'') &\asymp f(\omega') - f(\omega) \asymp D^{\delta_1},\\
g(\tau'') &\asymp f(\tau') - f(\tau) \asymp D^{\delta_2}
\end{align*}
and thus
\[
g(\tau'') \asymp g(\omega'')^{\delta_2/\delta_1}.
\]
Rewriting this using the definition of $g$, we see that
\[
4^{n!} \asymp 4^{s m!}
\] 
where $s = \delta_2/\delta_1$, $\sigma_n = \omega''$, and $\sigma_m = \tau''$.

Note that since $\lim_{b \in E} \| u_b' \| = 0$, Claim \ref{claim4points} implies that $\lim_{b \in E} f(\omega') = f (\omega)$, or equivalently that $\lim_{b \in E} |\omega'' | = \infty$. Thus $n$ and $m$ will both become arbitrarily large as $b$ ranges over $E$. If $n > m$ for infinitely many $b \in E$, then $4^{n!} \geq 4^{n m!} \gtrgtr 4^{s m!}$, a contradiction. Similarly, if $m \geq n$ for infinitely many $b \in E$, then since $s = \delta_2/\delta_1 >1$, we have $4^{m!} \geq 4^{n!} \gtrgtr 4^{s^{-1} n!}$, another contradiction. This concludes the proof of Theorem \ref{theoremcompactperfect}.

\draftnewpage
\section{Proof of Theorem \ref{theoremHDzero}}

Let $\UU = (u_a)_{a\in E}$ be a collection of similarities satisfying the OSC such that for all $a\in E := \N$,
\[
|u_a'| = 2^{-a^2}
\]
(the precise choice of these similarities does not matter as long as they satisfy the OSC). Let $A := \{0,1\}$, and let $f:A^*\cup A^\N \to \DS(\UU)$ be defined by the formula $f(\tau) := \HD(\Lambda_{A_\tau})$, where $A_\tau := \{a\in E : \tau_a = 1\}$.

Next we prove an analogue of Claim \ref{claim4points}, where the implied constant is now independent of the unperturbed limit set.

\begin{claim}
\label{claimbranch}
Fix $a_1,a_2 \in E$, and define $\w A^* := \{\omega \in A^* : \omega_{a_1} = \omega_{a_2} = 1\}$ and $\w A^\N := \{\omega \in A^\N : \omega_{a_1} = \omega_{a_2} = 1\}$. 
Then for all $\omega\in \w A^*$, we have
\[
f(\omega 1) = f(\omega 0) + \Theta(2^{-|\omega|^2 f(\omega)}),
\]
where the implied constants may depend on $a_1,a_2$.
\end{claim}
\begin{proof}
Let $\omega\in \w A^*$. Then it follows from Bowen's formula \cite[Theorem 4.2.11]{MauldinUrbanski2} that
\[
\sum_{a : \omega_n = 1} 2^{-a^2 f(\omega0)} = 1
\]
and that 
\[
\sum_{a : \omega_n = 1} 2^{-a^2 f(\omega1)} + 2^{-(|\omega| + 1)^2 f(\omega1)} = 1
\]
Hence,
\begin{align*}
2^{-(|\omega| + 1)^2 f(\omega1)} 
&= \sum_{a : \omega_n = 1} 2^{-a^2 f(\omega0)} - \sum_{a : \omega_n = 1} 2^{-a^2 f(\omega1)}\\
&= \left( f(\omega1) - f(\omega0) \right) \cdot  \sum_{a : \omega_n = 1}  \ln(2) 2^{-a^2 \xi} a^2 
\end{align*}
where $\xi \in (f(\omega0), f(\omega1))$. It then follows that
\[
2^{-(|\omega| + 1)^2 f(\omega1)} \asymp f(\omega1) - f(\omega0)
\]
since
\[
0 < 2^{-a_1^2} a_1^2 + 2^{-a_2^2} a_2^2 \leq \sum_{a : \omega_n = 1} 2^{-a^2 \xi} a^2 \leq \sum_{a \in E} 2^{-a^2 s} a^2 < \infty
\]
where $s := \HD(\{ a_1,a_2\}) > 0$. This concludes the proof of Claim \ref{claimbranch}.
\end{proof}

It then follows from Claim \ref{claimbranch} that for $\omega,\tau\in \w A^*$, we have
\[
|f(\omega) - f(\tau)| = O(2^{-|\omega\wedge\tau|^2 f(\omega)}) = O(2^{-s|\omega\wedge\tau|^2}),
\]
where $s = \HD(\{ a_1,a_2\})$. This implies that the box dimension, and thus that Hausdorff dimension, of $f(\w A^\N)$ is zero. By countable stability of Hausdorff dimension, 
\[
\HD(\DS(\UU)) = f(A^\N) = 0.
\]
This concludes the proof of Theorem \ref{theoremHDzero}.

\draftnewpage
\section{Conjectures and future work}

Our investigations of topological and metric properties of the dimension spectra of various conformal iterated function systems led to the following conjectures. 

\begin{conjecture}
The only sets $K \subset [0,\infty)$ such that both $K$ and its mirror image $\sup(K) - K$ are dimension spectra of CIFSes are intervals, i.e. $K = [0,\lambda]$ for some $\lambda >0$.
\end{conjecture}

Let us recall the definition of local Hausdorff dimension in the setting of a metric space: 
\begin{definition}
\label{localdimension}
Let $X$ be a metric space and let $F \subset X$.
For all $x \in X$ we define the \textsf{\textbf{local Hausdorff dimension}} of $F$ at $x$ as
\[
\dim_x(F) := \inf\{ \HD(F \cap B(x,\epsilon)) : \epsilon > 0 \}.
\]
The monotonicity of the Hausdorff dimension implies that the infimum in the definition above is actually a limit as $\epsilon$ tends to zero, i.e. 
\[
\dim_x(F) = \lim_{\epsilon \to 0} \HD(F \cap B(x,\epsilon)).
\]
\end{definition}

\begin{conjecture}
Let $F \subset [0,\infty)$ be the dimension spectrum of a CIFS.
The map $x \mapsto \dim_x(F)$ restricted to $x \in F$ is a continuous, weakly decreasing (i.e. nonincreasing) function.
\end{conjecture}

\begin{conjecture}
\label{conjecturetypes}
Let $F \subset [0,\infty)$ be the dimension spectrum of a CIFS.
Then one of the following three mutually exclusive scenarios holds:
\begin{itemize}
\item[Type I] $F$ is equal to the union of finitely many intervals.
\item[Type II] $F$ has zero Hausdorff dimension.
\item[Type III] The local Hausdorff dimension satisfies $\dim_x(F) = \min(1,c/x)$ for all $x \in F$, for some $0<c< \sup(F)$. I.e. the graph of the function $F\ni x\mapsto \dim_x(F)$ is a horizontal line followed by a hyperbola.
\end{itemize}
\end{conjecture}

\begin{remark}
For each of the three scenarios in Conjecture \ref{conjecturetypes} there exists a set $K$ exemplifying the scenario, which can be realized as the dimension spectrum of a SIFS. Indeed, this observation led us to Conjecture \ref{conjecturetypes}. Theorem \ref{theoremHDzero} provides an example of Type II. 
We leave it as an exercise for the interested reader to verify that the dimension spectrum of any SIFS whose similarities have contraction ratios $1/2,1/4,1/8,\ldots$ is of Type I, and similarly that the spectrum of one whose similarities have contraction ratios $1/3,1/3,1/9,1/27,\ldots$ is of Type III.
\end{remark}

\begin{remark}
If Conjecture \ref{conjecturetypes} is true then for each $F \subset [0,\infty)$ that is a dimension spectrum of a CIFS, the local Hausdorff dimension satisfies $\dim_x(F) = \min(1,c/x)$ for all $x \in F$, for some $0\leq c \leq \sup(F)$. The cases when $c=\sup(F)$ and $c=0$ correspond to Types I and II, respectively.
\end{remark}

In general, it appears difficult to distinguish between sets that can be SIFS or CIFS dimension spectra and those that cannot. In particular, it would be interesting to understand when an SIFS dimension spectrum could be realized as that of a CIFS that is not an SIFS, and vice versa.

The study of finite SIFSes with overlaps has witnessed several breakthroughs in the last decade, \cite{Hochman3, Hochman4}. It would be interesting to know whether dimension spectra behave differently in the absence of the OSC. 
For instance, recall that CLU proved \cite[Theorem 1.2]{CLU} that the dimension spectrum of an infinite conformal iterated function system satisfying the open set condition is compact and perfect.
However, this theorem is false for some systems that satisfy all conditions of being a CIFS except for the OSC.
Indeed, take the family of maps $\UU = \{ u_a:\R\to\R\}_{a \in E}$ defined by $u_a(x) := (1/2)x + a$ for $a \in E := \Q\cap [0,1]$. 
Then for any $F \subseteq E$, the dimension of $\Lambda_F$ is either $0$ or $1$ depending on whether or not $\#(F) \geq 2$, and thus $\DS(\UU) = \{0,1\}$.

Beyond similarity and conformal IFSes, the dimension spectra of affine IFSes remain unanalyzed. It may be fruitful to first focus on \emph{infinitely generated versions} of certain well-studied classes of finitely generated affine or other non-conformal IFSes, see e.g., \cite{Baranski, ChenPesin, DasSimmons1, Reeve}. See \cite{Jurga} for some recent progress in this direction.

In a different direction, rather than focussing on solely the Hausdorff dimension spectra, the study of spectra of other fractal dimensions -- such as packing dimension, box dimension, and Assouad dimension -- also awaits investigation.

\draftnewpage
\bibliographystyle{amsplain}
\bibliography{bibliography}

\end{document}